\pdfoutput=1
\documentclass[11pt]{amsart}

\usepackage[T1]{fontenc}
\usepackage{amsmath,amssymb,amsthm}
\usepackage{enumerate}
\usepackage[hidelinks]{hyperref}
\newcommand{\Pow}{\mathcal{P}}
\newcommand{\Diag}{\Delta}
\newcommand{\sU}{\mathcal{U}}
\newcommand{\sV}{\mathcal{V}}

\newcommand{\SE}{\operatorname{SE}}
\newcommand{\Rel}{\operatorname{Rel}_{s}}
\newcommand{\sprod}{\mathbin{\times_{s}}}
\newcommand{\TU}{\mathcal{T}_{\sU}}
\newcommand{\TV}{\mathcal{T}_{\sV}}
\newcommand{\Tprod}{\mathcal{T}_{\mathrm{prod}}}
\newcommand{\Tbox}{\mathcal{T}_{\mathrm{box}}}
\providecommand{\coloneqq}{\mathrel{:=}}

\theoremstyle{plain}
\newtheorem{theorem}{Theorem}
\newtheorem{lemma}[theorem]{Lemma}
\newtheorem{proposition}[theorem]{Proposition}
\theoremstyle{definition}
\newtheorem{definition}[theorem]{Definition}
\newtheorem{example}[theorem]{Example}

\hypersetup{
  pdftitle={Soft Uniform Spaces and Soft Uniform Continuity: Induced Topologies, Separation, and Compactness-Type Results},
  pdfauthor={S. Ray},
  pdfsubject={General Topology; soft uniform spaces},
  pdfkeywords={soft set, soft relation, soft uniformity, soft-element topology, uniform continuity, compactness, completeness}
}

\title[Soft Uniform Spaces and Soft Uniform Continuity]
{Soft Uniform Spaces and Soft Uniform Continuity:\\
Induced Topologies, Separation, and Compactness-Type Results}

\author{S. Ray}
\address{Department of Mathematics, Visva-Bharati University, Santiniketan, India}
\email{subhasis.ray@visva-bharati.ac.in}
\date{July 2026}

\begin{document}

\begin{abstract}
Soft uniform structures provide a way to describe uniform closeness in a parameterized setting.
Over a fixed parameter set, entourages are treated as soft relations, and a soft uniformity is defined by axioms parallel to the classical entourage axioms.
Each soft uniformity induces two related topologies: a sectionwise soft topology \(\tau_{\sU}\) on \(F\), and an ordinary soft-element topology \(\TU\) on \(\SE(F)\).
Both constructions are described. The space \((\SE(F),\TU)\) is Hausdorff, equivalently \(T_1\), exactly when the soft uniformity is separated, and it is regular for every soft uniformity.
Soft uniformly continuous mappings are then studied, and a Heine--Cantor type theorem is proved when the soft-element topology of the domain is compact.
Finally, total boundedness and completeness are formulated for the canonical soft-element uniform space \((\SE(F),\sU_{\SE})\), and compactness of \((\SE(F),\TU)\) is shown to imply both properties.
Examples relate the theory to uniformities generated by classical structures and illustrate the role of the parameter set.
\end{abstract}

\maketitle

\section{Introduction}

Soft set theory was introduced by Molodtsov \cite{Molodtsov1999} as a parameterized method for studying uncertainty. 
In a soft set, each parameter determines a subset of a fixed universe, and the whole object is handled through this family of sections. 
The basic operations on soft sets were developed by Maji, Biswas and Roy \cite{Maji2003}. 
Later work introduced many familiar mathematical structures in soft form, including soft groups \cite{Aktas2007}, soft topological spaces \cite{Shabir2011}, and soft metric spaces \cite{Das2013}. 
The language of soft elements gives another useful way to connect soft structures with ordinary mathematical structures \cite{GoldarRay2017,GoldarRay2018,GoldarRay2019}.

Uniform spaces provide a general setting for uniform continuity, Cauchy filters, completeness and total boundedness. 
They contain metric uniformities as important examples, but they do not require a particular metric \cite{Bourbaki1966}. 
It is therefore natural to study uniform ideas for soft sets.

Soft uniform structures have already appeared in several forms. 
For example, \"{O}zbak\i r and Demir \cite{OzbakirDemir2015} studied soft uniformity using soft points, and Demir, \"{O}zbak\i r and Y\i ld\i z \cite{Demir2016} considered fixed soft element theorems in SE-uniform spaces. Related approaches based on product soft sets and soft uniform neighbourhoods were developed by \"{O}zt\"{u}rk \cite{Ozturk2015} and \"{O}zt\"{u}rk--Yolcu \cite{Ozturk2016}.

In this paper we work over a fixed parameter set \(E\). 
A soft relation \(R\) on a soft set \(F\) is given coordinatewise by
\[
 R(e)\subseteq F(e)\times F(e),\qquad e\in E.
\]
Such a relation also determines an ordinary relation on the set
\[
 \SE(F)=\prod_{e\in E}F(e)
\]
of all soft elements. 
We denote this induced relation by
\[
 \SE(R)=\{(\mathbf{x},\mathbf{y})\in\SE(F)\times\SE(F):
(x_e,y_e)\in R(e)\text{ for every }e\in E\}.
\]
This passage from \(R\) to \(\SE(R)\) is useful because uniform continuity, convergence and Cauchy filters naturally involve pairs of soft elements.

We define a soft uniformity as a proper filter of soft relations satisfying the usual diagonal, inverse and square-root conditions. 
Each parameter section then carries an ordinary uniformity. 
At the same time, the relations \(\SE(U)\), where \(U\) ranges over the soft entourages, generate an ordinary uniformity on \(\SE(F)\). 
This gives two related topological objects: the induced soft topology on \(F\), and the soft-element topology on \(\SE(F)\).

Separation and regularity results are then established for these induced topologies. 
In particular, separatedness of the soft uniformity is equivalent to the \(T_1\) property of the induced soft-element topology, and the induced topology is regular. 
Soft uniformly continuous mappings are then considered. 
A soft mapping is soft uniformly continuous exactly when the induced map between soft-element uniform spaces is uniformly continuous.

Compactness is considered through the canonical soft-element topology. 
This compactness notion yields a soft Heine--Cantor theorem: a soft continuous mapping from a soft compact domain is soft uniformly continuous. 
Soft total boundedness is defined by finite soft entourage covers and soft completeness through soft Cauchy filters; soft compactness is shown to imply both properties.

The paper is organized as follows. 
Section~2 recalls soft sets, soft elements, soft relations and soft continuity. 
Section~3 develops soft uniformities and the induced topologies, including separation and regularity. 
Section~4 studies soft uniform continuity and proves a Heine--Cantor theorem. 
Section~5 treats soft total boundedness and soft completeness. 
Section~6 gives examples, and Section~7 concludes the paper.
\section{Preliminaries}

\subsection{Soft sets and soft elements}
Throughout, \(X\) is a nonempty universe and \(E\) is a nonempty parameter set.

\begin{definition}[Soft set {\cite{Molodtsov1999}}]
A \emph{soft set} over \(X\), with respect to \(E\), is a mapping \(F\colon E\to\Pow(X)\). We write \((F,E)\) for the corresponding soft set and call \(F(e)\) the \emph{\(e\)-section} of \(F\).
\end{definition}

\begin{definition}[Soft element {\cite{GoldarRay2018}}]
Let \((F,E)\) be a soft set. A \emph{soft element} of \(F\) is a mapping \(\mathbf{x}\colon E\to X\) such that \(x_e\coloneqq\mathbf{x}(e)\in F(e)\) for every \(e\in E\). The ordinary set of all soft elements is denoted by
\[
 \SE(F)\coloneqq\prod_{e\in E}F(e).
\]
Bold symbols \(\mathbf{x},\mathbf{y},\mathbf{z}\) always denote members of \(\SE(F)\), while \(x_e,y_e,z_e\) denote their coordinates.
\end{definition}

From this point onward we assume \(\SE(F)\neq\varnothing\) for every soft set under consideration. Fixing one element of \(\SE(F)\) then permits any prescribed value in one section to be extended to a soft element by leaving all other coordinates fixed.

\begin{definition}[Soft subset and soft product]
Let \(F\colon E\to\Pow(X)\) and \(G\colon E\to\Pow(Y)\) be soft sets over the same parameter set. A soft set \(A\) is a \emph{soft subset} of \(F\), written \(A\sqsubseteq F\), if \(A(e)\subseteq F(e)\) for every \(e\in E\). The \emph{soft product} \(F\sprod G\) is defined by
\[
 (F\sprod G)(e)=F(e)\times G(e),\qquad e\in E.
\]
Unions, intersections, and relative complements of soft subsets are taken parameterwise.
\end{definition}

For a soft subset \(A\sqsubseteq F\), we use the same soft-element notation
\[
 \SE(A)=\{\mathbf{x}\in\SE(F):x_e\in A(e)\text{ for every }e\in E\}
 =\prod_{e\in E}A(e).
\]
Thus \(\SE(A)\) is an ordinary subset of \(\SE(F)\). Not every subset of \(\SE(F)\) is of this rectangular form.

\subsection{Soft relations}
\begin{definition}[Soft relation]
A \emph{soft relation} on \(F\) is a soft subset \(R\sqsubseteq F\sprod F\); equivalently,
\[
 R(e)\subseteq F(e)\times F(e),\qquad e\in E.
\]
The set of all soft relations on \(F\), ordered by parameterwise inclusion, is denoted by \(\Rel(F)\).
\end{definition}

\begin{definition}[Diagonal, inverse, composition, and soft-element relation]
Let \(R,S\in\Rel(F)\).
\begin{enumerate}[(i)]
\item The diagonal soft relation is
\[
 \Diag_F(e)=\{(a,a):a\in F(e)\}.
\]
\item The inverse is \(R^{-1}(e)=\{(b,a):(a,b)\in R(e)\}\).
\item The composition is defined parameterwise by
\[
\begin{aligned}
 (a,c)\in(R\circ S)(e)
 &\quad\Longleftrightarrow\quad
 \text{there exists }b\in F(e)\text{ such that}\\
 &\hspace{42mm}(a,b)\in S(e)\text{ and }(b,c)\in R(e).
\end{aligned}
\]
\item The soft-element relation determined by \(R\) is the ordinary relation
\[
 \SE(R)=\{(\mathbf{x},\mathbf{y})\in\SE(F)^2:(x_e,y_e)\in R(e)\text{ for every }e\in E\}.
\]
\end{enumerate}
\end{definition}

\begin{lemma}
For \(R,S\in\Rel(F)\),
\begin{enumerate}[(i)]
\item \(\SE(\Diag_F)=\Diag_{\SE(F)}\);
\item \(\SE(R\cap S)=\SE(R)\cap\SE(S)\);
\item \(\SE(R^{-1})=\SE(R)^{-1}\);
\item \(\SE(R)\circ\SE(S)\subseteq\SE(R\circ S)\).
\end{enumerate}
\end{lemma}

\begin{proof}
The first three statements follow coordinate by coordinate. If \((\mathbf{x},\mathbf{z})\in \SE(R)\circ\SE(S)\), there is \(\mathbf{y}\in\SE(F)\) such that \((x_e,y_e)\in S(e)\) and \((y_e,z_e)\in R(e)\) for every \(e\). Hence \((x_e,z_e)\in(R\circ S)(e)\) for every \(e\), which proves (iv).
\end{proof}

\subsection{Soft topologies and soft continuity}
All operations on soft subsets are understood parameterwise.

\begin{definition}[Soft topology {\cite{Shabir2011}}]
A family \(\tau\) of soft subsets of \(F\) is a \emph{soft topology} on \(F\) if it contains the empty soft set and \(F\), is closed under arbitrary soft unions, and is closed under finite soft intersections. Its members are called soft open sets.
\end{definition}

\begin{definition}[Soft mapping and sectionwise soft continuity]
Let \(F\colon E\to\Pow(X)\) and \(G\colon E\to\Pow(Y)\). 
A \emph{soft mapping} \(f\colon F\to G\) over the fixed parameter set \(E\) is a family of mappings
\[
f_e\colon F(e)\to G(e),\qquad e\in E.
\]
It induces
\[
\SE(f)\colon\SE(F)\to\SE(G),\qquad
\SE(f)(\mathbf{x})_e=f_e(x_e).
\]

A soft topology \(\tau_F\) on \(F\) is called \emph{sectionwise} if, for each \(e\in E\), there is an ordinary topology \(\tau_F(e)\) on \(F(e)\) such that
\[
\tau_F
=
\{A\sqsubseteq F:A(e)\in\tau_F(e)
\text{ for every }e\in E\}.
\]
A sectionwise soft topology \(\tau_G\) on \(G\) is defined in the same way, using an ordinary topology \(\tau_G(e)\) on \(G(e)\) for each \(e\in E\).

The soft mapping \(f\) is \emph{sectionwise soft continuous} if each mapping
\[
f_e\colon
\bigl(F(e),\tau_F(e)\bigr)
\longrightarrow
\bigl(G(e),\tau_G(e)\bigr)
\]
is continuous. Equivalently, for every \(B\in\tau_G\), the soft inverse image defined by
\[
f^{-1}(B)(e)=f_e^{-1}(B(e)),\qquad e\in E,
\]
belongs to \(\tau_F\).
\end{definition}

The equivalence follows by testing one coordinate at a time. 
For \(O_e\in\tau_G(e)\), define \(B(e)=O_e\) and
\(B(d)=G(d)\) for every \(d\neq e\).

Throughout the paper, bold letters such as \(\mathbf{x}\) denote soft elements, while \(x_e\) denotes the \(e\)-coordinate of \(\mathbf{x}\).

\section{Soft Uniformities and the Induced Topologies}

\subsection{Soft uniformities}
The discussion begins with the filter structure.

\begin{definition}[Filter of soft relations]
A family \(\mathcal F\subseteq\Rel(F)\) is a \emph{proper filter of soft relations} if
\begin{enumerate}[(i)]
\item \(\mathcal F\neq\varnothing\) and the soft relation with all sections empty does not belong to \(\mathcal F\);
\item \(R,S\in\mathcal F\) implies \(R\cap S\in\mathcal F\);
\item \(R\in\mathcal F\) and \(R\sqsubseteq T\in\Rel(F)\) imply \(T\in\mathcal F\).
\end{enumerate}
\end{definition}

\begin{definition}[Soft uniformity and soft uniform space]
A \emph{soft uniformity} on \(F\) is a proper filter \(\sU\) of soft relations satisfying
\begin{enumerate}[(U1)]
\item \(\Diag_F\sqsubseteq U\) for every \(U\in\sU\);
\item \(U^{-1}\in\sU\) whenever \(U\in\sU\);
\item for every \(U\in\sU\) there exists \(V\in\sU\) such that \(V\circ V\sqsubseteq U\).
\end{enumerate}
The pair \((F,\sU)\) is a \emph{soft uniform space}, and members of \(\sU\) are called \emph{soft entourages}.
\end{definition}

Thus finite intersection and upward closure are part of the filter, while (U1)--(U3) are the diagonal, inverse, and square-root conditions.

\begin{definition}[Base of a soft uniformity]
A nonempty family \(\mathcal B\subseteq\Rel(F)\) is a \emph{base of soft entourages} if
\begin{enumerate}[(B1)]
\item \(\Diag_F\sqsubseteq B\) for every \(B\in\mathcal B\);
\item for \(B_1,B_2\in\mathcal B\) there exists \(B_3\in\mathcal B\) with \(B_3\sqsubseteq B_1\cap B_2\);
\item for every \(B\in\mathcal B\) there exists \(C\in\mathcal B\) with \(C^{-1}\sqsubseteq B\);
\item for every \(B\in\mathcal B\) there exists \(C\in\mathcal B\) with \(C\circ C\sqsubseteq B\).
\end{enumerate}
\end{definition}

\begin{proposition}
A nonempty family \(\mathcal B\subseteq\Rel(F)\) is a base of soft entourages if and only if
\[
 \sU_{\mathcal B}=\{U\in\Rel(F):B\sqsubseteq U\text{ for some }B\in\mathcal B\}
\]
is a soft uniformity. Every soft uniformity is a base for itself.
\end{proposition}

\begin{proof}
Assume (B1)--(B4). Conditions (B1) and (B2) give a proper filter whose members contain \(\Diag_F\). If \(B\sqsubseteq U\) and \(C^{-1}\sqsubseteq B\) as in (B3), then \(C\sqsubseteq U^{-1}\), so \(U^{-1}\in\sU_{\mathcal B}\). If \(C\circ C\sqsubseteq B\sqsubseteq U\), then \(C\) is a square-root of \(U\). Hence \(\sU_{\mathcal B}\) is a soft uniformity. Conversely, a soft uniformity itself satisfies (B1)--(B4) by the filter, inverse, and square-root axioms.
\end{proof}

\begin{proposition}
For each \(e\in E\), the family
\[
 \sU_e\coloneqq\{U(e):U\in\sU\}
\]
is a classical uniformity on \(F(e)\).
\end{proposition}

\begin{proof}
The diagonal, finite-intersection, inverse, and square-root properties pass to the \(e\)-section. For upward closure, suppose \(U(e)\subseteq H\subseteq F(e)^2\). Define \(W\in\Rel(F)\) by \(W(e)=H\) and \(W(d)=F(d)^2\) for \(d\neq e\). Then \(U\sqsubseteq W\), so \(W\in\sU\) and \(H=W(e)\in\sU_e\).
\end{proof}

The next theorem supplies a single uniform space on the set of soft elements.

\begin{theorem}
\label{thm:14}
Let \((F,\sU)\) be a soft uniform space and define
\[
 \sU_{\SE}=\{W\subseteq\SE(F)^2:\SE(U)\subseteq W\text{ for some }U\in\sU\}.
\]
Then \(\sU_{\SE}\) is a classical uniformity on \(\SE(F)\), with \(\{\SE(U):U\in\sU\}\) as an entourage base.
\end{theorem}

\begin{proof}
Every \(\SE(U)\) contains \(\Diag_{\SE(F)}\). The identity \(\SE(U\cap V)=\SE(U)\cap \SE(V)\) gives finite intersections, and upward closure is built into the definition. Also \(\SE(U^{-1})=(\SE(U))^{-1}\). Finally, if \(V\circ V\sqsubseteq U\), then
\[
 \SE(V)\circ \SE(V)\subseteq\SE(V\circ V)\subseteq \SE(U).
\]
These are the classical entourage axioms.
\end{proof}

\paragraph{Relation with earlier formulations.}
The present setting is close to the classical entourage description of a uniform space, but it is written for soft relations. 
In \cite{OzbakirDemir2015}, soft uniformity is developed by using one-parameter soft points. 
The approaches in \cite{Ozturk2015,Ozturk2016} use product soft sets and soft uniform neighbourhoods, and they also allow mappings between parameter sets. 
Here the parameter set is fixed, and a point of the space is a soft element \(\mathbf{x}\in\SE(F)\), that is, a simultaneous choice of one element from each section. 
This viewpoint leads to the ordinary uniformity \(\sU_{\SE}\) on \(\SE(F)\), which will be used below for continuity, compactness and Cauchy filters.

\subsection{Soft neighbourhoods and induced topology}

\begin{definition}[Soft entourage neighbourhood]
For \(U\in\sU\) and \(\mathbf{x}\in\SE(F)\), define the soft subset \(U[\mathbf{x}]\sqsubseteq F\) by
\[
 U[\mathbf{x}](e)=\{a\in F(e):(x_e,a)\in U(e)\}.
\]
Its realization in \(\SE(F)\) is
\[
 \SE(U[\mathbf{x}])=\SE(U)[\mathbf{x}]
 =\prod_{e\in E}U(e)[x_e].
\]
\end{definition}

\begin{definition}[Topologies induced by a soft uniformity]
Let \(\tau_e\) be the topology induced by \(\sU_e\) on \(F(e)\).
\begin{enumerate}[(i)]
\item The \emph{induced soft topology} is
\[
 \tau_{\sU}=\{O\sqsubseteq F:O(e)\in\tau_e\text{ for every }e\in E\}.
\]
\item The \emph{soft-element topology}, denoted by \(\TU\), is the ordinary topology induced by \(\sU_{\SE}\) on \(\SE(F)\).
\end{enumerate}
\end{definition}

The two objects have different types. The family \(\tau_{\sU}\) is a soft topology consisting of soft subsets of \(F\), whereas \(\TU\) is an ordinary topology on the set \(\SE(F)\). The topological separation, convergence, compactness, and completeness statements below refer to \((\SE(F),\TU)\); total boundedness and the Cauchy condition refer to its underlying uniformity \(\sU_{\SE}\).

\begin{theorem}
The family \(\tau_{\sU}\) is a soft topology on \(F\). Moreover, for each \(\mathbf{x}\in\SE(F)\), the family
\[
 \{\SE(U)[\mathbf{x}]:U\in\sU\}
\]
is a neighbourhood base at \(\mathbf{x}\) in \(\TU\).
\end{theorem}

\begin{proof}
The empty soft set and \(F\) belong to \(\tau_{\sU}\), and arbitrary unions and finite intersections remain open in every section; hence \(\tau_{\sU}\) is a soft topology. The second assertion is the standard neighbourhood-base description of the topology induced by the entourage base \(\{\SE(U):U\in\sU\}\) from Theorem~\ref{thm:14}.
\end{proof}

The two induced topologies are related through the section topologies.

\begin{theorem}
\label{thm:18}
Let \(\Tprod\) and \(\Tbox\) be, respectively, the product and box topologies on \(\SE(F)=\prod_{e\in E}F(e)\), formed from the topologies \(\tau_e\). Then
\[
 \Tprod\subseteq\TU\subseteq\Tbox.
\]
If \(E\) is finite, all three topologies coincide.
\end{theorem}

\begin{proof}
Let \(\mathbf{x}\) belong to a basic product neighbourhood \(N=\prod_eN_e\), where \(N_e=F(e)\) except for finitely many \(e\) in a set \(J\). For each \(e\in J\), choose \(U_e\in\sU\) with \(U_e(e)[x_e]\subseteq N_e\). Then \(U=\bigcap_{e\in J}U_e\) satisfies \(\SE(U)[\mathbf{x}]\subseteq N\), so every product-open set is \(\TU\)-open.

Conversely, let \(O\in\TU\) and \(\mathbf{x}\in O\). Choose \(U\in\sU\) with \(\SE(U)[\mathbf{x}]\subseteq O\). For each \(e\), the set \(U(e)[x_e]\) is a neighbourhood in \(\tau_e\); choose \(N_e\in\tau_e\) with \(x_e\in N_e\subseteq U(e)[x_e]\). Then the box-open set \(\prod_eN_e\) contains \(\mathbf{x}\) and is contained in \(O\). For finite \(E\), product and box topologies are equal.
\end{proof}

\subsection{Separatedness and regularity}

The separation and regularity properties in this subsection are properties of the ordinary soft-element space \((\SE(F),\TU)\).

\begin{definition}[Separated soft uniformity]
A soft uniformity \(\sU\) is \emph{separated} if
\[
 \bigcap_{U\in\sU}U=\Diag_F,
\]
where the intersection is taken parameterwise.
\end{definition}

\begin{theorem}
For a soft uniform space \((F,\sU)\), the following conditions are equivalent:
\begin{enumerate}[(i)]
\item \(\sU\) is separated;
\item every section uniformity \(\sU_e\) is separated;
\item the soft-element uniformity \(\sU_{\SE}\) is separated;
\item \((\SE(F),\TU)\) is Hausdorff;
\item \((\SE(F),\TU)\) is \(T_1\).
\end{enumerate}
\end{theorem}

\begin{proof}
The equivalence of (i) and (ii) follows directly from the parameterwise definition. Assume (i) and let \(\mathbf{x}\neq\mathbf{y}\). For some \(e_0\), \(x_{e_0}\neq y_{e_0}\), and separatedness gives \(U\in\sU\) with \((x_{e_0},y_{e_0})\notin U(e_0)\). Hence \((\mathbf{x},\mathbf{y})\notin \SE(U)\), so \(\sU_{\SE}\) is separated.

Conversely, if (i) fails, there are \(e_0\in E\) and distinct \(a,b\in F(e_0)\) such that \((a,b)\in U(e_0)\) for every \(U\in\sU\). Fix \(\mathbf{p}\in\SE(F)\) and define \(\mathbf{x},\mathbf{y}\) by \(x_{e_0}=a\), \(y_{e_0}=b\), and \(x_e=y_e=p_e\) for \(e\neq e_0\). Since every entourage contains the diagonal, \((\mathbf{x},\mathbf{y})\in \SE(U)\) for every \(U\), so \(\sU_{\SE}\) is not separated. Thus (i) and (iii) are equivalent. The equivalence of (iii), (iv), and (v) is the usual separation theorem for uniform spaces.
\end{proof}

Each section topology is regular by the classical uniform-space theorem. The next result concerns the ordinary topology \(\TU\) on the full soft-element set \(\SE(F)\).

\begin{theorem}
Let \((F,\sU)\) be a soft uniform space. If \(\mathbf{x}\in O\in\TU\), then there exists \(G\in\TU\) such that
\[
 \mathbf{x}\in G\subseteq\operatorname{cl}_{\TU}G\subseteq O.
\]
Consequently \((\SE(F),\TU)\) is regular, and it is \(T_3\) when \(\sU\) is separated.
\end{theorem}

\begin{proof}
Choose \(U\in\sU\) with \(\SE(U)[\mathbf{x}]\subseteq O\). By the square-root axiom, choose \(W\in\sU\) such that \(W\circ W\sqsubseteq U\), and put
\[
 V=W\cap W^{-1}.
\]
Then \(V\in\sU\), \(V=V^{-1}\), and
\[
 V\circ V\sqsubseteq W\circ W\sqsubseteq U.
\]
Since \(\Diag_F\sqsubseteq V\), we also have
\[
 \SE(V)[\mathbf{x}]
 \subseteq \SE(V\circ V)[\mathbf{x}]
 \subseteq \SE(U)[\mathbf{x}]
 \subseteq O.
\]
Put
\[
 G=\operatorname{int}_{\TU}\bigl(\SE(V)[\mathbf{x}]\bigr).
\]
Because \(\SE(V)[\mathbf{x}]\) is a neighbourhood of \(\mathbf{x}\), one has \(\mathbf{x}\in G\).

We next show that
\[
 \operatorname{cl}_{\TU}\bigl(\SE(V)[\mathbf{x}]\bigr)
 \subseteq \SE(U)[\mathbf{x}].
\]
Let \(\mathbf{y}\notin \SE(U)[\mathbf{x}]\). If some \(\mathbf{z}\) belonged to both \(\SE(V)[\mathbf{x}]\) and \(\SE(V)[\mathbf{y}]\), then symmetry of \(V\) would give
\[
 (\mathbf{x},\mathbf{y})\in \SE(V)\circ \SE(V)
 \subseteq\SE(V\circ V)\subseteq \SE(U),
\]
a contradiction. Hence \(\SE(V)[\mathbf{y}]\) is a neighbourhood of \(\mathbf{y}\) disjoint from \(\SE(V)[\mathbf{x}]\). Its \(\TU\)-interior is an open neighbourhood of \(\mathbf{y}\) with the same property. Thus
\(
 \mathbf{y}\notin\operatorname{cl}_{\TU}(\SE(V)[\mathbf{x}]).
\)
It follows that
\[
 \operatorname{cl}_{\TU}G
 \subseteq \operatorname{cl}_{\TU}\bigl(\SE(V)[\mathbf{x}]\bigr)
 \subseteq \SE(U)[\mathbf{x}]
 \subseteq O.
\]

If \(C\) is closed and \(\mathbf{x}\notin C\), apply the construction to \(O=\SE(F)\setminus C\). Then \(G\) and
\[
 O_2=\SE(F)\setminus\operatorname{cl}_{\TU}G
\]
are disjoint open sets containing \(\mathbf{x}\) and \(C\), respectively.
\end{proof}

\section{Soft Uniform Continuity}

\subsection{Definition and basic properties}
We now compare the relation-based definition with the ordinary uniformity on soft elements.

\begin{definition}[Soft uniformly continuous mapping]
Let \((F,\sU)\) and \((G,\sV)\) be soft uniform spaces. A soft mapping \(f=\{f_e\}_{e\in E}\colon F\to G\) is \emph{soft uniformly continuous} if for every \(V\in\sV\) there exists \(U\in\sU\) such that
\[
 (f\sprod f)(U)\sqsubseteq V,
\]
where
\[
 (f\sprod f)(U)(e)=\{(f_e(a),f_e(b)):(a,b)\in U(e)\}.
\]
The mapping is called \emph{soft-element continuous} if \(\SE(f)\colon(\SE(F),\TU)\to(\SE(G),\TV)\) is continuous.
\end{definition}

\begin{theorem}
\label{thm:23}
A soft mapping \(f\colon(F,\sU)\to(G,\sV)\) is soft uniformly continuous if and only if
\[
 \SE(f)\colon(\SE(F),\sU_{\SE})\longrightarrow(\SE(G),\sV_{\SE})
\]
is uniformly continuous in the classical sense.
\end{theorem}

\begin{proof}
Assume \(f\) is soft uniformly continuous. Given \(V\in\sV\), choose \(U\in\sU\) with \((f\sprod f)(U)\sqsubseteq V\). Then
\[
 (\SE(f)\times \SE(f))(\SE(U))\subseteq \SE(V),
\]
so \(\SE(f)\) is uniformly continuous.

Conversely, suppose \(\SE(f)\) is uniformly continuous and fix \(V\in\sV\). Choose \(W\in\sU_{\SE}\) with \((\SE(f)\times \SE(f))(W)\subseteq \SE(V)\), and choose \(U\in\sU\) with \(\SE(U)\subseteq W\). Fix \(e\in E\) and \((a,b)\in U(e)\). Take \(\mathbf{p}\in\SE(F)\) and define \(\mathbf{x},\mathbf{y}\) by \(x_e=a\), \(y_e=b\), and \(x_d=y_d=p_d\) for \(d\neq e\). Since \(U\) contains the diagonal, \((\mathbf{x},\mathbf{y})\in \SE(U)\subseteq W\). Hence \((f_e(a),f_e(b))\in V(e)\). This holds for every \(e\), proving \((f\sprod f)(U)\sqsubseteq V\).
\end{proof}

\begin{proposition}
Every soft uniformly continuous mapping is soft-element continuous and sectionwise soft continuous.
\end{proposition}

\begin{proof}
The first assertion follows from Theorem~\ref{thm:23} and the classical implication from uniform continuity to continuity. For a fixed \(e\), every target entourage in \(\sV_e\) is the \(e\)-section of some \(V\in\sV\); the corresponding source entourage supplied by soft uniform continuity shows that \(f_e\) is uniformly continuous. Thus each \(f_e\) is continuous, which is sectionwise soft continuity.
\end{proof}

\begin{proposition}
The composition of two soft uniformly continuous mappings is soft uniformly continuous.
\end{proposition}

\begin{proof}
The induced soft-element map of a composition is the composition of the induced maps. The result follows from Theorem~\ref{thm:23} and the corresponding classical fact.
\end{proof}

When \(E\) is finite, Theorem~\ref{thm:18} shows that soft-element continuity is equivalent to sectionwise soft continuity. For infinitely many parameters, soft-element continuity may be stronger because \(\TU\) can be strictly finer than the product topology.

\subsection{A soft Heine--Cantor theorem}
We now introduce the compactness notion used in the uniform results below.

\begin{definition}[Soft compactness]
A soft uniform space \((F,\sU)\) is \emph{soft compact} if the canonical space \((\SE(F),\TU)\) is compact.
\end{definition}

This compactness notion is the one naturally associated with the uniformity \(\sU_{\SE}\) on the soft-element space. 
It is related to the soft-element approach to compactness studied in \cite{GoldarRay2019}, where a soft subset \(A\) is considered through the ordinary set \(\SE(A)\) of its soft elements. 
In the present paper the topology is the canonical topology \(\TU\) induced by the soft uniformity.

This differs from merely requiring compactness of all parameter sections. 
Indeed, the induced maps \(\SE(f)\), the entourages \(\SE(U)\), and the Cauchy filters used later are all defined on \(\SE(F)\). 
The next proposition records the relation with parameterwise compactness.

\begin{proposition}
If \((F,\sU)\) is soft compact, then every section \((F(e),\tau_e)\) is compact. If \(E\) is finite, the converse also holds. For an infinite parameter set, the converse can fail.
\end{proposition}

\begin{proof}
The coordinate projection \(\pi_e\colon(\SE(F),\TU)\to(F(e),\tau_e)\) is continuous because \(\Tprod\subseteq\TU\) by Theorem~\ref{thm:18}. It is surjective: a prescribed \(a\in F(e)\) can be extended to a soft element by fixing the remaining coordinates. Hence compactness of \(\SE(F)\) implies compactness of every section. If \(E\) is finite, \(\TU=\Tprod\), so compactness of all sections implies compactness of \(\SE(F)\). The failure for infinite \(E\) is shown in Example~\ref{ex: 36}.
\end{proof}

\begin{lemma}
\label{lem:28}
Let \((F,\sU)\) be soft compact and let \(\mathcal C\) be an open cover of \((\SE(F),\TU)\). There exists \(U\in\sU\) such that, for every \(\mathbf{x}\in\SE(F)\), the neighbourhood \(\SE(U)[\mathbf{x}]\) is contained in some member of \(\mathcal C\).
\end{lemma}

\begin{proof}
For each \(\mathbf{x}\), choose \(C_{\mathbf{x}}\in\mathcal C\) containing \(\mathbf{x}\). Choose \(H_{\mathbf{x}}\in\sU\) with \(\SE(H_{\mathbf{x}})[\mathbf{x}]\subseteq C_{\mathbf{x}}\), and then choose a symmetric \(V_{\mathbf{x}}\in\sU\) such that \(V_{\mathbf{x}}\circ V_{\mathbf{x}}\sqsubseteq H_{\mathbf{x}}\). The open sets
\[
 G_{\mathbf{x}}=\operatorname{int}_{\TU}\bigl(\SE(V_{\mathbf{x}})[\mathbf{x}]\bigr)
\]
cover \(\SE(F)\). Choose a finite subcover \(G_{\mathbf{x}_1},\ldots,G_{\mathbf{x}_n}\) and put \(U=\bigcap_{k=1}^{n}V_{\mathbf{x}_k}\).

Given \(\mathbf{x}\), choose \(k\) with \(\mathbf{x}\in G_{\mathbf{x}_k}\). If \(\mathbf{y}\in \SE(U)[\mathbf{x}]\), then \((\mathbf{x}_k,\mathbf{x})\in \SE(V_{\mathbf{x}_k})\) and \((\mathbf{x},\mathbf{y})\in \SE(V_{\mathbf{x}_k})\). Therefore
\[
 (\mathbf{x}_k,\mathbf{y})\in \SE(V_{\mathbf{x}_k})\circ \SE(V_{\mathbf{x}_k})
 \subseteq \SE(H_{\mathbf{x}_k}),
\]
so \(\mathbf{y}\in C_{\mathbf{x}_k}\). Thus \(\SE(U)[\mathbf{x}]\subseteq C_{\mathbf{x}_k}\).
\end{proof}

\begin{theorem}
Let \((F,\sU)\) be soft compact, let \((G,\sV)\) be a soft uniform space, and let \(f\colon F\to G\) be soft-element continuous. Then \(f\) is soft uniformly continuous.
\end{theorem}

\begin{proof}
Let \(V\in\sV\). Choose a symmetric \(W\in\sV\) with \(W\circ W\sqsubseteq V\). For each \(\mathbf{x}\in\SE(F)\), the set
\[
 C_{\mathbf{x}}=(\SE(f))^{-1}\!\left(\operatorname{int}_{\TV}\SE(W)[\SE(f)(\mathbf{x})]\right)
\]
is open and contains \(\mathbf{x}\). These sets cover \(\SE(F)\). By Lemma~\ref{lem:28}, choose \(U\in\sU\) such that every \(\SE(U)[\mathbf{x}]\) lies in one of them.

If \((\mathbf{x},\mathbf{y})\in \SE(U)\), then \(\mathbf{x}\) and \(\mathbf{y}\) lie in a common \(C_{\mathbf{z}}\). Hence both images lie in \(\SE(W)[\SE(f)(\mathbf{z})]\), and symmetry gives
\[
 (\SE(f)(\mathbf{x}),\SE(f)(\mathbf{y}))\in \SE(W)\circ \SE(W)\subseteq \SE(V).
\]
Thus \(\SE(f)\) is uniformly continuous. By Theorem~\ref{thm:23} the proof follows.
\end{proof}

\section{Soft Total Boundedness and Soft Completeness}
The global definitions are taken in the canonical uniform space \((\SE(F),\sU_{\SE})\).

\begin{definition}[Soft total boundedness]
A soft uniform space \((F,\sU)\) is \emph{soft totally bounded} if, for every \(U\in\sU\), there exist \(\mathbf{x}_1,\ldots,\mathbf{x}_n\in\SE(F)\) such that
\[
 \SE(F)=\bigcup_{k=1}^{n}\SE(U)[\mathbf{x}_k].
\]
Equivalently, the soft-element uniform space \((\SE(F),\sU_{\SE})\) is totally bounded.
\end{definition}

\begin{definition}[Soft Cauchy filter and completeness]
A \emph{soft Cauchy filter} is an ordinary proper filter \(\mathfrak F\) on \(\SE(F)\) such that, for every \(U\in\sU\), there exists \(M\in\mathfrak F\) satisfying
\[
 M\times M\subseteq \SE(U).
\]
The soft uniform space is \emph{soft complete} if every soft Cauchy filter converges in \(\TU\).
\end{definition}

Thus a soft Cauchy filter is a Cauchy filter in the ordinary uniform space \((\SE(F),\sU_{\SE})\). 
The term ``soft'' reflects that the entourages used in the Cauchy condition come from the soft uniformity. 
This is different from intrinsic soft filters, which are families of soft subsets; such filters are studied in \cite{GoldarRaySarkar2025Filters}. 
Since an ordinary filter on \(\SE(F)\) may contain nonrectangular subsets, this formulation is suited to the soft-element uniformity used in this paper.

\begin{theorem}
If a soft uniform space \((F,\sU)\) is soft compact, then it is soft totally bounded and soft complete.
\end{theorem}

\begin{proof}
Let \(U\in\sU\). The interiors of \(\SE(U)[\mathbf{x}]\), as \(\mathbf{x}\) ranges over \(\SE(F)\), form an open cover. A finite subcover gives
\[
 \SE(F)=\bigcup_{k=1}^{n}\SE(U)[\mathbf{x}_k],
\]
so the space is soft totally bounded.

Let \(\mathfrak F\) be a soft Cauchy filter. Compactness gives a cluster point \(\mathbf{p}\) of \(\mathfrak F\), since the closures of its members have the finite-intersection property. Fix \(U\in\sU\) and choose a symmetric \(V\in\sU\) with \(V\circ V\sqsubseteq U\). There exists \(M\in\mathfrak F\) with \(M\times M\subseteq \SE(V)\). Since \(\mathbf{p}\) is a cluster point, \(M\cap \SE(V)[\mathbf{p}]\neq\varnothing\); choose \(\mathbf{q}\) in this intersection. For every \(\mathbf{y}\in M\),
\[
 (\mathbf{p},\mathbf{q})\in \SE(V),\qquad (\mathbf{q},\mathbf{y})\in \SE(V),
\]
so \((\mathbf{p},\mathbf{y})\in \SE(U)\). Thus \(M\subseteq \SE(U)[\mathbf{p}]\), and upward closure gives \(\SE(U)[\mathbf{p}]\in\mathfrak F\). This holds for every \(U\), so \(\mathfrak F\) converges to \(\mathbf{p}\).
\end{proof}

\section{Examples}

\begin{example}[Discrete soft uniformity]
Let \(F\) be a soft set with \(\SE(F)\neq\varnothing\) and define
\[
 \sU_{\mathrm{disc}}=\{U\in\Rel(F):\Diag_F\sqsubseteq U\}.
\]
This is a soft uniformity. Since \(\Diag_F\in\sU_{\mathrm{disc}}\), the soft-element topology on \(\SE(F)\) is discrete. Every soft mapping from this space into any soft uniform space is soft uniformly continuous: for a target entourage, use \(\Diag_F\) in the domain.
\end{example}

\begin{example}[Soft uniformity induced by a family of metrics]
For each \(e\in E\), let \(d_e\) be a metric on \(F(e)\). For \(\varepsilon>0\), set
\[
 H_{\varepsilon}(e)=\{(a,b)\in F(e)^2:d_e(a,b)<\varepsilon\}.
\]
The family \(\{H_{\varepsilon}:\varepsilon>0\}\) is a base of soft entourages, since \(H_{\varepsilon/2}\circ H_{\varepsilon/2}\sqsubseteq H_{\varepsilon}\). Every section has its metric topology, while a basic soft-element ball is
\[
 \SE(H_{\varepsilon})[\mathbf{x}]=\prod_{e\in E}B_{d_e}(x_e,\varepsilon),
\]
so one common \(\varepsilon\) controls all parameters. For finite \(E\), this is the uniformity of the maximum metric on \(\prod_eF(e)\).
\end{example}

\begin{example}[A two-parameter compact soft uniform space]
Let \(E=\{1,2\}\), let \(F(1)=F(2)=[0,1]\), and use the metric-generated soft uniformity of the preceding example. Then \(\SE(F)=[0,1]^2\) with the maximum-metric topology, hence \(F\) is soft compact. Define \(f_1(t)=t^2\) and \(f_2(t)=t^3\). For \(a,b\in[0,1]\),
\[
 |a^2-b^2|\leq2|a-b|,\qquad |a^3-b^3|\leq3|a-b|.
\]
Thus \(H_{\varepsilon/3}\) in the domain is mapped into \(H_{\varepsilon}\) in the codomain. The example shows a single entourage that controls two different coordinate mappings.
\end{example}

\begin{example}[Sectionwise compactness is not sufficient]
\label{ex: 36}
Let \(E=\mathbb N\) and \(F(n)=\{0,1\}\) for every \(n\). Give \(F\) the discrete soft uniformity. Each section is a finite compact discrete space. However,
\[
 \SE(F)=\{0,1\}^{\mathbb N}
\]
has the discrete topology because \(\SE(\Diag_F)=\Diag_{\SE(F)}\) is an entourage. Since \(\SE(F)\) is infinite, it is not compact. Hence compactness of all parameter sections does not imply soft compactness when the parameter set is infinite.
\end{example}

\begin{example}[Soft continuous but not soft uniformly continuous]
Let \(E=\{e_0\}\), \(F(e_0)=(0,4)\), and \(G(e_0)=(0,\infty)\), with their usual metric uniformities. Define \(f_{e_0}(x)=1/x\). The map is soft-element continuous, which here is ordinary continuity, but it is not soft uniformly continuous. Indeed, \(|1/n-2/n|=1/n\to0\), while
\[
 \left|f_{e_0}(1/n)-f_{e_0}(2/n)\right|=n/2.
\]
\end{example}

\section{Conclusion}

The paper develops a relation-based framework for soft uniform spaces and distinguishes the two topologies that arise from it. The family \(\tau_{\sU}\) is a sectionwise soft topology on \(F\), while \(\TU\) is the ordinary topology induced on the soft-element set \(\SE(F)\). The separation and regularity results concern \((\SE(F),\TU)\): this space is Hausdorff, equivalently \(T_1\), exactly when the soft uniformity is separated, and it is regular for every soft uniformity.

The paper also studies soft uniform continuity and proves a Heine--Cantor type theorem when the domain space \((\SE(F),\TU)\) is compact. Soft total boundedness and soft completeness were formulated as total boundedness and completeness of the canonical uniform space \((\SE(F),\sU_{\SE})\). Under these definitions, compactness of \((\SE(F),\TU)\) implies both properties.

There are several natural continuations. 
A completion theory for soft uniform spaces, including existence, uniqueness and functoriality, would be valuable. 
It would also be useful to study product and quotient constructions for soft uniformities in a systematic way. 
Finally, since our completeness notion is filter-based, it would be interesting to compare it in detail with recent work on soft filters, soft nets and soft ideals in \cite{GoldarRaySarkar2025Filters,GoldarRaySarkar2025Ideals}, and to clarify when parameterwise completeness can be recovered from soft completeness.

\end{document}